\documentclass[twoside,12pt,leqno]{article}
\advance\topmargin by -\headheight\advance\topmargin by -\headsep
\newcommand{\tmvolxx}{xx}
\newcommand{\tmyearyyyy}{yyyy}

\newcommand{\FirstPageHead}[3]{
{\footnotesize 
\vskip -8mm 
\centerline {Travaux math\'ematiques, \quad 
Volume #1 (#2), 
#3,\quad \copyright\  Universit\'e du Luxembourg}}\vspace{-3mm}}

\usepackage{amsmath}
\usepackage{amsthm}
\usepackage{amssymb}
\usepackage{amscd}
\usepackage{graphicx}
\usepackage{epsfig}

\newtheorem{theorem}{Theorem}[section]
\newtheorem{lemma}[theorem]{Lemma}
\newtheorem{proposition}[theorem]{Proposition}
\newtheorem{corollary}[theorem]{Corollary}

\theoremstyle{remark}
\newtheorem{remark}[theorem]{Remark}

\numberwithin{equation}{section}

\allowdisplaybreaks

\begin{document}

\thispagestyle{empty}
\FirstPageHead{\tmvolxx}{\tmyearyyyy}{\pageref{firstpage}--\pageref{lastpage}}
\label{firstpage}
{\renewcommand{\thefootnote}{\fnsymbol{footnote}}
\setcounter{footnote}{0}
\footnotetext{
Higgs $G$--bundle, co--Higgs bundle, Yang--Mills--Higgs connection,
representation space, deformation retraction.}
\setcounter{footnote}{1}
}
{\renewcommand{\thefootnote}{\fnsymbol{footnote}}
\setcounter{footnote}{0}

\footnotetext{AMS Classification: 14F05, 14J32, 32L05, 53C07, 58E15}
\setcounter{footnote}{1}
}




\markboth{I. Biswas, U. Bruzzo, B. Gra\~na Otero, A. Lo Giudice}
{Yang--Mills--Higgs connections on Calabi--Yau manifolds, II}

$ $
\bigskip

$ $
\bigskip

\bigskip

\centerline{{\Large {Yang--Mills--Higgs
connections on Calabi--Yau manifolds, II}  }}

\bigskip
\bigskip
\centerline{{\large by  
Indranil Biswas, Ugo Bruzzo,
}}
\medskip
\centerline{{\large 
Beatriz Gra\~na Otero, and
Alessio Lo Giudice
}
\addtocounter{footnote}{-1}\footnote
{I.B. thanks Universit\'e Pierre et Marie Curie for hospitality.
He is supported by a J. C. Bose Fellowship. U.B.'s
research is partly supported by 
INdAM-GNSAGA. B.G.O. is supported by Project ID PRY 6777 of the Pontificia Universidad
Javeriana, Bogot\'a. U.B. is a member of VBAC.}
}

\vspace*{.7cm}

\begin{abstract}
In this paper we study Higgs and co--Higgs $G$--bundles on compact K\"ahler manifolds $X$. Our main 
results are:

\smallskip
\begin{enumerate}
\item If $X$ is Calabi--Yau (i.e., it has vanishing first Chern class), and $(E,\,\theta)$ is
a semistable Higgs or co--Higgs $G$--bundle on $X$, then the principal $G$--bundle $E$ is semistable.
In particular, there is a 
deformation retract of ${\mathcal M}_H(G)$ onto $\mathcal M(G)$, where $\mathcal M(G)$ is the moduli 
space of semistable principal $G$--bundles with vanishing rational Chern classes on $X$, and analogously, 
${\mathcal M}_H(G)$ is the moduli space of semistable principal Higgs $G$--bundles with vanishing 
rational Chern classes.

\smallskip
\item Calabi--Yau manifolds are characterized as those compact K\"ahler 
manifolds whose tangent bundle is semistable for every K\"ahler
class, and have the following property: if $(E,\,\theta)$ is a 
semistable Higgs or co--Higgs vector bundle, then $E$ is semistable.
\end{enumerate}
\end{abstract}

\pagestyle{myheadings}

\section{Introduction}

In our previous paper \cite{BBGL} we showed that the existence of semistable Higgs bundles with a nontrivial Higgs field
on a compact K\"ahler manifold $X$ constrains the geometry of $X$. In particular, it was shown that if $X$ is K\"ahler-Einstein with $c_1(TX)\ge 0$, then it is
necessarily Calabi-Yau, i.e., $c_1(TX)=0$. In this paper we extend the
analysis of the interplay between the existence of semistable Higgs bundles
and the geometry of the underlying manifold (actually, we shall also consider co-Higgs bundles, and allow the structure group of the bundle
to be any reductive linear algebraic group). Thus, 
if  $X$ is Calabi--Yau  and $(E,\,\theta)$ is
a semistable Higgs or co--Higgs $G$--bundle on $X$, 
it is proved that the underlying principal
$G$--bundle $E$ is semistable (Lemma \ref{lem2}).
This has a consequence on the topology of the moduli spaces of principal (Higgs) $G$-bundles having vanishing
rational Chern classes. We can indeed prove that   there is a 
deformation retract of ${\mathcal M}_H(G)$ onto $\mathcal M(G)$, where $\mathcal M(G)$ is the moduli 
space of semistable principal $G$--bundles with vanishing rational Chern classes, and analogously, 
${\mathcal M}_H(G)$ is the moduli space of semistable principal Higgs $G$--bundles with vanishing 
rational Chern classes (cf.~\cite{BF,FL} for similar deformation retract results).

As a further application, we can prove a characterization of Calabi--Yau manifolds in terms of
Higgs and co-Higgs bundles; the characterization in question says that if $X$ is a compact K\"ahler 
manifold with semistable tangent bundle with respect to every K\"ahler
class, having the following property: for any
semistable Higgs or co--Higgs vector bundle $(E,\,\theta)$  on $X$, the vector bundle $E$ is
semistable, then $X$ is Calabi-Yau (Theorem \ref{prop5}).
 
In Section \ref{se4}
We give a result about the behavior of semistable Higgs bundles under pullback 
by finite morphisms of K\"ahler manifolds. Let $(X,\, \omega)$ be a Ricci--flat
compact K\"ahler manifold, $M$ a compact connected K\"ahler manifold,
and
$$
f\, :\, M\, \longrightarrow\, X
$$
a surjective holomorphic map such that each fiber of $f$ is a finite subset of $M$.
Let $(E_G,\, \theta)$ be a Higgs $G$--bundle on $X$ such that the pulled back
Higgs $G$--bundle $(f^*E_G,\, f^*\theta)$ on $M$ is semistable (respectively,
stable). Then the principal $G$--bundle $f^*E_G$ is semistable (respectively,
polystable).

\section{Preliminaries}

Let $X$ be a compact connected K\"ahler manifold equipped with a K\"ahler form $\omega$. 
Using $\omega$, the degree of torsion-free coherent analytic sheaves on $X$ is defined
as follows:
$$
\text{degree}(F)\, :=\, \int_X c_1(F)\bigwedge \omega^{d-1}\,\in\, \mathbb R\, ,
$$
where $d\,=\, \dim_{\mathbb C}X$. The holomorphic cotangent bundle of $X$ will be denoted by $\Omega_X$.

Let $G$ be a connected reductive affine algebraic group
defined over $\mathbb C$. The connected component of the center of $G$ containing the
identity element will be denoted by $Z_0(G)$.
The Lie algebra of $G$ will be denoted by $\mathfrak g$.
A Zariski closed connected subgroup $P\, \subseteq\, G$ is called a parabolic subgroup
of $G$ if $G/P$ is a projective variety. The unipotent radical of a parabolic subgroup
$P$ will be denoted by $R_u(P)$. A Levi subgroup of a parabolic subgroup
$P$ is a Zariski closed subgroup $L(P)\, \subset\, P$ such that the composition
$$
L(P)\, \hookrightarrow\, P\, \longrightarrow\, P/R_u(P)
$$
is an isomorphism. Levi subgroups exist,
and any two Levi subgroups of $P$ differ by an inner automorphism of $P$ \cite[\S~11.22, 
p.~158]{Bo}, \cite[\S~30.2, p.~184]{Hu2}. The quotient map
$G\, \longrightarrow\, G/P$ defines a principal $P$--bundle on $G/P$.
The holomorphic line bundle on $G/P$ associated to this principal $P$--bundle for a
character $\chi$ of $P$ will be denoted by $G(\chi)$.
A character $\chi$ of a parabolic subgroup $P$ 
is called \textit{strictly anti--dominant} if $\chi\vert_{Z_0(G)}$ is trivial, and the 
associated holomorphic line bundle on $G(\chi)\, \longrightarrow\, G/P$ is ample.

For a principal $G$--bundle $E_G$ on $X$, the vector bundle
$$
\text{ad}(E_G)\, :=\, E_G\times^G{\mathfrak g} \, \longrightarrow\, X
$$
associated to $E_G$ for the adjoint action of $G$ on
its Lie algebra $\mathfrak g$ is called the
\textit{adjoint bundle} for $E_G$. So the fibers of $\text{ad}(E_G)$ are Lie algebras
identified with $\mathfrak g$ up to inner automorphisms.
Using the Lie algebra structure of the fibers of $\text{ad}(E_G)$ and the
exterior multiplication of differential forms we have a symmetric bilinear pairing
$$
(\text{ad}(E_G)\otimes\Omega_X)\times (\text{ad}(E_G)\otimes\Omega_X)\,\longrightarrow\,
\text{ad}(E_G)\otimes\Omega^2_X
$$
which will be denoted by $\bigwedge$.

A \textit{Higgs field} on a holomorphic principal $G$--bundle $E_G$ on $X$ is a holomorphic
section $\theta$ of $\text{ad}(E_G)\otimes\Omega_X$ such that
\begin{equation}\label{e1}
\theta\bigwedge\theta \,=\, 0\, .
\end{equation}
A \textit{Higgs} $G$--{\it bundle} on $X$ is a pair of the form $(E_G,\, \theta)$, where
$E_G$ is holomorphic principal $G$--bundle on $X$ and 
$\theta$ is a Higgs field on $E_G$. A Higgs $G$--bundle $(E_G,\, \theta)$ is called
\textit{stable} (respectively, \textit{semistable}) if for every quadruple of the form
$(U,\, P,\, \chi,\,  E_P)$, where
\begin{itemize}
\item $U\, \subset\, X$ is a dense open subset such that the
complement $X\setminus U$ is a complex
analytic subset of $X$ of complex codimension at least two,

\item $P\, \subset\, G$ is a proper parabolic
subgroup,

\item $\chi$ is a strictly anti--dominant character of $P$, and

\item $E_P\, \subset\, E_G\vert_U$
is a holomorphic reduction of structure group to $P$ over $U$
such that $\theta\vert_U$ is a section of $\text{ad}(E_P)\otimes\Omega_U$,
\end{itemize}
the following holds:
$$
\text{degree}(E_P\times^\chi \mathbb C)\, >\, 0
$$
(respectively, $\text{degree}(E_P\times^\chi \mathbb C)\, \geq \, 0$);
note that since $X\setminus U$ is a complex
analytic subset of $X$ of complex codimension at least two, the line bundle
$E_P\times^\chi \mathbb C$ on $U$ extends uniquely to a holomorphic line bundle on $X$.

A semistable Higgs $G$--bundle $(E_G,\, \theta)$ is called
\textit{polystable} if there is a Levi subgroup $L(Q)$ of a parabolic subgroup $Q\,
\subset\, G$ and a Higgs $L(Q)$--bundle $(E',\, \theta')$ on $X$ such that
\begin{itemize}
\item the Higgs $G$--bundle obtained by extending the structure group of $(E',\, \theta')$
using the inclusion $L(Q)\, \hookrightarrow\, G$ is isomorphic to $(E_G,\, \theta)$, and

\item the Higgs $L(Q)$--bundle $(E',\, \theta')$ is stable.
\end{itemize}

Fix a maximal compact subgroup $K\, \subset\, G$. Given a holomorphic principal $G$--bundle
$E_G$ and a $C^\infty$ reduction of structure group $E_K\, \subset\, E_G$ to
the subgroup $K$, there is a unique connection on the principal $K$--bundle
$E_K$ that is compatible with the holomorphic structure of $E_G$
\cite[pp.~191--192, Proposition~5]{At}; it is known as the \textit{Chern connection}.
A $C^\infty$ reduction of structure group of $E_G$ to $K$ is called a
\textit{Hermitian structure} on $E_G$.

Let $\Lambda_\omega$ denote the adjoint of multiplication of differential forms on $X$
by $\omega$.

Given a Higgs $G$--bundle $(E_G,\, \theta)$ on $X$, a Hermitian structure $E_K\,\subset\,
E_G$ is said to satisfy the Yang--Mills--Higgs equation if
\begin{equation}\label{ymh}
\Lambda_\omega ({\mathcal K}+ \theta\bigwedge\theta^*) \,=\, {\mathfrak z}\, ,
\end{equation}
where $\mathcal K$ is the curvature of the Chern connection
associated to $E_K$ and $\mathfrak z$ is some
element of the Lie algebra of $Z_0(G)$; the adjoint $\theta^*$ is
constructed using the Hermitian structure $E_K$. A Higgs $G$--bundle admits a Hermitian
structure satisfying the Yang--Mills--Higgs equation if and only if it is polystable
\cite{Si}, \cite[p. 554, Theorem 4.6]{BS}.

Given a polystable Higgs $G$--bundle $(E,\, \theta)$, any two Hermitian structures on
$E_G$ satisfying the Yang--Mills--Higgs equation differ by a holomorphic automorphism of
$E_G$ that preserves $\theta$; however, the associated Chern connection is unique
\cite[p. 554, Theorem 4.6]{BS}.

\section{Higgs $G$--bundles on Calabi--Yau manifolds}

Henceforth, till the end of Section \ref{se4}, we assume that 
$c_1(TX)\,\in\,H^2(X, \, {\mathbb Q})$ is zero. From this assumption
it follows that every K\"ahler class on $X$ contains
a Ricci--flat K\"ahler metric \cite[p. 364, Theorem 2]{Ya}. We will assume that the
K\"ahler form $\omega$ on $X$ is Ricci--flat.

\subsection{Higgs $G$--bundles on Calabi-Yau manifolds}

Let $(E_G,\, \theta)$ be a polystable Higgs $G$--bundle on $X$. For any
holomorphic tangent vector $v\, \in\, T_xX$, note that $\theta(x)$ is an
element of the fiber $\text{ad}(E_G)_x$. For any point
$x\, \in\, X$, consider the complex subspace
\begin{equation}\label{ht}
\widehat{\Theta}_x\, :=\, \{\theta(x)(v)\, \mid\, v\, \in \, T_xX\}\, \subset\,
\text{ad}(E_G)_x\, .
\end{equation}
{}Form \eqref{e1} it follows immediately that $\widehat{\Theta}_x$ is an abelian
subalgebra of the Lie algebra $\text{ad}(E_G)_x$.

Let $\nabla$ be the connection on $\text{ad}(E_G)$ induced by the unique
connection on $E_G$ given by the solutions of the Yang--Mills--Higgs equation.

\begin{lemma}\label{lem1}
\mbox{}
\begin{enumerate}
\item The abelian subalgebra $\widehat{\Theta}_x\, \subset\,{\rm ad}(E_G)_x$ is semisimple.

\item $\{\widehat{\Theta}_x\}_{x\in X}\,\subset\, {\rm ad}(E_G)$ 
is preserved by the connection $\nabla$ on ${\rm ad}(E_G)$. In particular,
$$
\{\widehat{\Theta}_x\}_{x\in X}\,\subset\, {\rm ad}(E_G)
$$
is a holomorphic subbundle.
\end{enumerate}
\end{lemma}

\begin{proof}
First take $G\,=\, \text{GL}(n, {\mathbb C})$, so that $(E_G,\, \theta)$ defines a
Higgs vector bundle $(F,\, \varphi)$ of rank $n$. Let
$$
\widehat{\Theta}'_x\, \subset\, \text{End}(F_x)
$$
be the subalgebra constructed as in \eqref{ht} for the Higgs
vector bundle $(F,\, \varphi)$. From
\cite[Proposition 2.5]{BBGL} it follows immediately that there is a basis of the vector
space $F_x$ such that $$\varphi(x)(v)\, \in\, \text{End}(F_x)$$ is diagonal with respect
to it for all $v\,\in\, T_x$. This implies that the subalgebra $\widehat{\Theta}'_x$ is
semisimple (this uses the Jordan--Chevalley decomposition, see e.g.\ \cite[Ch.\ 2]{Hu1}).

Consider the ${\mathcal O}_X$--linear homomorphism
$$
\eta\, :\, TX\, \longrightarrow\, \text{End}(F)
$$
that sends any $w\, \in\, T_yX$ to $\varphi(y)(w)\, \in\, \text{End}(F_y)$, where
$\varphi$ as before is the Higgs field on the holomorphic vector bundle $F$. Proposition
2.2 of \cite{BBGL} says that $\varphi$ is flat with respect to the connection on
$\text{End}(F)\otimes\Omega_X$ induced by the connection $\nabla$ on $\text{End}(F)
\,=\, \text{ad}(E_G)$ and the Levi--Civita connection on $\Omega_X$ for
$\omega$. Therefore, the above homomorphism $\eta$ intertwines the Levi--Civita connection
on $TX$ and the connection on $\text{End}(F)$. Consequently, the image $\eta(TX)\,\subset\,
\text{End}(F)$ is preserved by the connection on $\text{End}(F)$. On the other hand,
$\eta(TX)$ coincides with $\{\widehat{\Theta}'_x\}_{x\in X}\,\subset\, \text{End}(F)$.

Therefore, the lemma is proved when $G\,=\, \text{GL}(n, {\mathbb C})$. 

For a general $G$, take any homomorphism
$$
\rho\, :\, G\,\longrightarrow\, \text{GL}(N, {\mathbb C})
$$ 
such that $\rho(Z_0(G))$ lies inside the center of $\text{GL}(N, {\mathbb C})$. Let
$(E_\rho,\, \theta_\rho)$ be the Higgs vector bundle of rank $N$ given by
$(E_G,\, \theta)$ using $\rho$. For any Hermitian structure on $E_G$ solving the
Yang--Mills--Higgs equation for $(E_G,\, \theta)$, the induced Hermitian structure
on $E_\rho$ solves the Yang--Mills--Higgs equation for $(E_\rho,\, \theta_\rho)$.
We have shown above that the lemma holds for $(E_\rho,\, \theta_\rho)$.

Since the lemma holds for $(E_\rho,\, \theta_\rho)$ for every $\rho$ of the above type,
we conclude that the lemma holds for $(E_G,\, \theta)$.
\end{proof}

As before, $(E_G,\, \theta)$ is a polystable Higgs $G$--bundle on $X$. Fix
a Hermitian structure
\begin{equation}\label{ek}
E_K\, \subset\, E_G
\end{equation}
that satisfies the Yang--Mills--Higgs equation for $(E_G,\, \theta)$.

Take another
Higgs field $\beta$ on $E_G$. Let
$$
\widetilde{\beta}\, :\, TX\, \longrightarrow\, \text{ad}(E_G)
$$
be the ${\mathcal O}_X$--linear homomorphism that sends any tangent
vector $w\, \in\, T_yX$ to $$\beta(y)(w)\, \in\, \text{ad}(E_G)_y\, .$$

\begin{theorem}\label{thm1}
Assume that the image $\widetilde{\beta}(TX)$ is contained in the subbundle
$$\{\widehat{\Theta}_x\}_{x\in X}\,\subset\, {\rm ad}(E_G)$$ in Lemma \ref{lem1}.
Then $E_K$ in \eqref{ek} also satisfies the Yang--Mills--Higgs equation for
$(E_G,\, \beta)$. In particular, $(E_G,\, \beta)$ is polystable.
\end{theorem}

\begin{proof}
{}From Theorem 4.2 of \cite{BBGL} we know that $E_K$ in \eqref{ek}
satisfies the Yang--Mills--Higgs equation for
$(E_G,\, 0)$. Therefore, it suffices to show that $\beta\bigwedge\beta^*\,=\,0$
(see \eqref{ymh}).

Let
$$
\gamma\, :\, TX\, \longrightarrow\, \text{ad}(E_G)
$$
be the $C^\infty(X)$--linear homomorphism that sends any $w\, \in \, T_yX$ to
$\theta^*(y)(w)\, \in\, \text{ad}(E_G)_y$. Clearly, we have
\begin{equation}\label{i1}
\gamma(TX)^*\,=\, \{\widehat{\Theta}_x\}_{x\in X}\, ;
\end{equation}
as before, the superscript ``$*$'' denotes adjoint with respect to the reduction
$E_K$. Since the subbundle $\{\widehat{\Theta}_x\}_{x\in X}$ is preserved by the
connection on $\text{ad}(E_G)$, from \eqref{i1} it follows that
\begin{equation}\label{i2}
\{\widehat{\Theta}_x\}_{x\in X} + \gamma(TX)\, \subset\, \text{ad}(E_G)
\end{equation}
is a subbundle preserved by the connection; it should be clarified that the above
need not be a direct sum.

We know that $\theta\bigwedge\theta^*\,=\, 0$ \cite[Lemma 4.1]{BBGL}. This and
\eqref{e1} together imply that the subbundle in \eqref{i2} is an abelian subalgebra
bundle. We have $$\widetilde{\beta}(TX)\, \subset\, \{\widehat{\Theta}_x\}_{x\in X}\, ,$$
and hence $\beta^*$ is a section of $\gamma(TX)\otimes\Omega_X\, \subset\,
\text{ad}(E_G)\otimes \Omega_X$. Since the subbundle in \eqref{i2} is an abelian
subalgebra bundle, we now conclude that $\beta\bigwedge\beta^*\,=\, 0$.
\end{proof}

The proof of Theorem \ref{thm1} gives the following:

\begin{corollary}\label{cor1}
Let $\phi$ be a Higgs field on $E_G$ such that $\phi\bigwedge\phi^*\,=\, 0$. 
Then $E_K$ in \eqref{ek} also satisfies the Yang--Mills--Higgs equation for
$(E_G,\, \phi)$. In particular, $(E_G,\, \phi)$ is polystable.
\end{corollary}

\begin{proof}
As noted in the proof of Theorem \ref{thm1}, the reduction $E_K$ satisfies the
Yang--Mills--Higgs equation for $(E_G,\, 0)$. Since $\phi\bigwedge\phi^*\,=\, 0$, it
follows that $E_K$ in \eqref{ek} satisfies the Yang--Mills--Higgs equation
for $(E_G,\, \phi)$.
\end{proof}

\begin{remark}\label{rem1}
The condition in Theorem \ref{thm1} that $\widetilde{\beta}(TX)\, \subset\, 
\{\widehat{\Theta}_x\}_{x\in X}$ does not depend on the Hermitian structure $E_K$; it 
depends only on the Higgs $G$--bundle $(E_G,\, \theta)$. In contrast, the condition 
$\phi\bigwedge\phi^*\,=\, 0$ in Corollary \ref{cor1} depends also on $E_K$.
\end{remark}

\subsection{A deformation retraction}

Let ${\mathcal M}_H(G)$ denote the moduli space of semistable Higgs $G$--bundles $(E_G,\,
\theta)$ on $X$ such that all  rational characteristic classes of $E_G$ of positive
degree vanish. It is known (it is a straightforward consequence of Theorem 2 in \cite{Si}) that if the following three conditions hold:
\begin{enumerate}
\item $(E_G,\, \theta)$ is semistable,

\item for all characters $\chi$ of $G$, the line bundle on $X$ associated to
$E_G$ for $\chi$ is of degree zero, and

\item the second Chern class $c_2(\text{ad}(E_G))\,\in\, H^4(X,\, {\mathbb Q})$
vanishes,
\end{enumerate}
then all characteristic classes of $E_G$ of positive degree vanish. Let
${\mathcal M}(G)$ denote the moduli space of semistable principal $G$--bundles $E_G$ on
$X$ such that all   rational characteristic classes of $E_G$ of positive degree vanish.

We have an inclusion
\begin{equation}\label{xi}
\xi\, :\, {\mathcal M}(G)\, \longrightarrow\, {\mathcal M}_H(G)\, ,\ \
E_G\, \longmapsto\, (E_G,\, 0)\, .
\end{equation}

\begin{proposition}\label{prop3}
There is a natural holomorphic deformation retraction of ${\mathcal M}_H(G)$ to the image
of $\xi$ in \eqref{xi}.
\end{proposition}

\begin{proof}
Points of ${\mathcal M}_H(G)$ parametrize the polystable Higgs $G$--bundles
$(E_G,\, \theta)$ on $X$ such that all   rational characteristic classes of
$E_G$ of positive degree vanish. Given such a Higgs $G$--bundle $(E_G,\, \theta)$, from
Theorem \ref{thm1} we know that $(E_G,\, t\cdot \theta)$ is polystable for
all $t\, \in\, \mathbb C$. Therefore, we have a holomorphic map
$$
F\, :\, {\mathbb C}\times {\mathcal M}_H(G)\, \longrightarrow\, {\mathcal M}_H(G)\, ,
\ \ (t,\, (E_G,\, \theta))\, \longmapsto\, (E_G,\, t\cdot \theta)\, .
$$
The restriction of $F$ to $\{1\}\times {\mathcal M}_H(G)$ is the identity map
of ${\mathcal M}_H(G)$, while the image of the restriction of $F$ to
$\{0\}\times {\mathcal M}_H(G)$ is the image of $\xi$. Moreover,
the restriction of $F$ to $\{0\}\times \xi ({\mathcal M}(G))$ is the identity map.
\end{proof}

Fix a point $x_0\, \in\, X$. Since $G$ is an affine variety and $\pi_1(X,\, x_0)$
is finitely presented, the geometric invariant theoretic quotient
$${\mathcal M}_R(G)\,:=\, \text{Hom}(\pi_1(X,\, x_0),\, G)/\!\!/ G$$
for the adjoint action of $G$ on $\text{Hom}(\pi_1(X,\, x_0),\, G)$ is an affine variety.
The points of ${\mathcal M}_R(G)$ parameterize the equivalence classes of homomorphisms from
$\pi_1(X,\, x_0)$ to $G$ such that the Zariski closure of the image is a reductive subgroup
of $G$. Consider the quotient space
$${\mathcal M}_R(K)\,:=\, \text{Hom}(\pi_1(X,\, x_0),\, K)/K\, ,$$
where $K$ as before is a maximal compact subgroup of $G$. The inclusion of $K$ in $G$
produces an inclusion
\begin{equation}\label{xi2}
\xi'\, :\, {\mathcal M}_R(K)\, \longrightarrow\, {\mathcal M}_R(G)\, .
\end{equation}

\begin{corollary}\label{prop4}
There is a natural deformation retraction of ${\mathcal M}_R(G)$ to
the subset ${\mathcal M}_R(K)$ in \eqref{xi2}.
\end{corollary}

\begin{proof}
The nonabelian Hodge theory gives a homeomorphism of ${\mathcal M}_R(G)$ with
${\mathcal M}_H(G)$. On the other hand, ${\mathcal M}_R(K)$ is identified
with ${\mathcal M}(G)$, and the following diagram is commutative:
$$
\begin{matrix}
{\mathcal M}(G) & \stackrel{\xi}{\longrightarrow} & {\mathcal M}_H(G)\\
~\Big\downarrow \sim && ~\Big\downarrow\sim \\
{\mathcal M}_R(K)& \stackrel{\xi'}{\longrightarrow} & {\mathcal M}_R(G)
\end{matrix}
$$
Hence Proposition \ref{prop3} produces the deformation retraction in question.
\end{proof}

\section{Pullback of Higgs bundles by finite morphisms}\label{se4}

Take $(X,\, \omega)$ to be as before. Let $M$ be compact connected K\"ahler manifold,
and let
$$
f\, :\, M\, \longrightarrow\, X
$$
be a surjective holomorphic map such that each fiber of $f$ is a finite subset of $M$.
In particular, we have $\dim M\,=\, \dim X$. It is known that the form
$f^*\omega$ represents a K\"ahler class on the
K\"ahler manifold $M$ \cite[p. 438, Lemma 2.1]{BiSu}.
The degree of torsion-free coherent analytic sheaves on $M$ will be defined using
the K\"ahler class given by $f^*\omega$.

\begin{proposition}\label{prop1}
Let $(E_G,\, \theta)$ be a Higgs $G$--bundle on $X$ such that the pulled back
Higgs $G$--bundle $(f^*E_G,\, f^*\theta)$ on $M$ is semistable. Then the principal
$G$--bundle $f^*E_G$ is semistable.
\end{proposition}

\begin{proof}
Since the pulled back
Higgs $G$--bundle $(f^*E_G,\, f^*\theta)$ is semistable, it follows that
$(E_G,\, \theta)$ is semistable. Indeed, the pullback of any reduction of
structure group of $(E_G,\, \theta)$ that contradicts the semistability condition
also contradicts the semistability condition for $(f^*E_G,\, f^*\theta)$.
Since the Higgs $G$--bundle $(E_G,\, \theta)$ is semistable, we conclude that the principal $G$--bundle
$E_G$ is semistable \cite[p. 305, Lemma 6.2]{Bi}. This, in turn, implies that $f^*E_G$ is
semistable (see \cite[p. 441, Theorem 2.4]{BiSu} and \cite[p. 442, Remark 2.5]{BiSu}).
\end{proof}

\begin{proposition}\label{prop2}
Let $(E_G,\, \theta)$ be a Higgs $G$--bundle on $X$ such that the pulled back
Higgs $G$--bundle $(f^*E_G,\, f^*\theta)$ on $M$ is stable. Then the principal
$G$--bundle $f^*E_G$ is polystable.
\end{proposition}

\begin{proof}
The principal $G$--Higgs bundle $(E_G,\, \theta)$ is stable, because any reduction of it 
contradicting the stability condition pulls back to a reduction that contradicts the 
stability condition for $(f^*E_G,\, f^*\theta)$. Since $(E_G,\, \theta)$ is stable, we 
know that $E_G$ is polystable \cite[p. 306, Lemma 6.4]{Bi}. Now $f^*E_G$ is polystable 
because $E_G$ is so \cite[p.~439, Proposition~2.3]{BS}, \cite[p.~442, Remark~2.6]{BS}.
\end{proof}

\section{Co--Higgs bundles}

We recall the definition of a co--Higgs vector bundle \cite{Ra1,Ra2,Hi}.

Let $(X,\,\omega)$ be a compact connected K\"ahler manifold and $E$ a holomorphic vector bundle on
$X$. A \textit{co--Higgs field} on $E$ is a holomorphic section
$$
\theta \, \in\, H^0(X,\, \text{End}(E)\otimes TX)
$$
such that the section $\theta\bigwedge\theta$ of $\text{End}(E)\otimes \bigwedge^2 TX$ vanishes
identically. A co--Higgs bundle on $X$ is a pair $(E, \, \theta)$, where $E$ is a holomorphic
vector bundle on $X$ and $\theta$ is a co--Higgs field on $E$ \cite{Ra1,Ra2,Hi}.

A co--Higgs bundle $(E, \, \theta)$ is called \textit{semistable} if for all nonzero
coherent analytic subsheaves $F\, \subset\, E$ with $\theta(F)\, \subset\, F\otimes TX$, the
inequality
$$
\mu(F)\, :=\, \frac{\text{degree}(F)}{\text{rank}(F)}\, \leq\,
\frac{\text{degree}(E)}{\text{rank}(E)}\, :=\, \mu(E)
$$
holds.

\subsection{Co--Higgs bundles on Calabi--Yau manifolds}

In this subsection we assume that $c_1(TX)\,\in\, H^2(X, \, {\mathbb Q})$ is zero, and
the K\"ahler form $\omega$ on $X$ is Ricci--flat. Take a holomorphic vector bundle
$E$ on $X$.

\begin{lemma}\label{lem2}
Let $\theta$ be a Higgs field or a co--Higgs field on $E$ such that $(E,\, \theta)$ is
semistable. Then the vector bundle $E$ is semistable.
\end{lemma}

\begin{proof}
Let $\theta$ be a co--Higgs field on $E$ such that the co--Higgs bundle $(E,\, \theta)$
is semistable. Assume that $E$ is not semistable. Let $F$ be the maximal semistable
subsheaf of $E$, in other words, $F$ is the first term in the Harder--Narasimhan filtration
of $E$. The maximal semistable
subsheaf of $E/F$ will be denoted by $F_1$, so $\mu_{\rm max}(E/F)\,=\, \mu(F_1)$. Note that
we have \begin{equation}\label{m1}
\mu(F)\, >\, \mu(F_1) \,=\, \mu_{\rm max}(E/F)\, .
\end{equation}

Since $\omega$ is Ricci--flat we know that $TX$ is polystable. The tensor product of a
semistable sheaf and a semistable vector bundle is semistable \cite[p.~212, Lemma~2.7]{AB}.
Therefore, the maximal semistable subsheaf of $(E/F)\otimes TX$ is
$$F_1\otimes TX\, \subset\, (E/F)\otimes TX\, .$$ Now,
$$
\mu(F_1\otimes TX)\,=\, \mu(F_1)
$$
because $c_1(TX)\,=\, 0$. Hence from \eqref{m1} it follows that
\begin{equation}\label{eco}
\mu(F)\, >\, \mu(F_1\otimes TX)\,=\, \mu_{\rm max}((E/F)\otimes TX)\, .
\end{equation}

Let
$$
q\, :\, E\, \longrightarrow\, E/F
$$
be the quotient homomorphism. From \eqref{eco} it follows that
there is no nonzero homomorphism from $E$ to $(E/F)\otimes TX$. In particular,
the composition
$$
F\, \hookrightarrow\, E \, \stackrel{\theta}{\longrightarrow}\, E\otimes TX
\, \stackrel{q\otimes {\rm Id}}{\longrightarrow}\, (E/F)\otimes TX
$$
vanishes identically. This immediately implies that $\theta(F)\, \subset\, F\otimes TX$.
Therefore, the co--Higgs subsheaf $(F,\, \theta\vert_F)$ of $(E,\, \theta)$ violates the inequality
in the definition of semistability. But this
contradicts the given condition that $(E,\, \theta)$ is semistable. Hence we conclude
that $E$ is semistable.

Note that $\Omega_X$ is polystable because $TX$ is polystable. Hence the above proof also works
when the co-Higgs field $\theta$ is replaced by a Higgs field.
\end{proof}

A particular case of this result was shown in \cite{Ra2} for $X$ a K3 surface. Moreover,
a result implying this Lemma was proved in \cite{BH}.

\subsection{A characterization of Calabi--Yau manifolds}

\begin{theorem}\label{prop5}
Let $X$ be a compact connected K\"ahler manifold such that for every K\"ahler class $[\omega]
\,\in\, H^2(X,\, {\mathbb R})$ on it the following two hold:
\begin{enumerate}
\item the tangent bundle $TX$ is semistable, and

\item for every semistable Higgs or co--Higgs bundle $(E,\,\theta)$ on $X$,
the underlying holomorphic vector bundle $E$ is semistable.
\end{enumerate}
Then $c_1(TX)\,=\, 0$.
\end{theorem}

\begin{proof}
We will show that $\text{degree}(TX)\,=\, 0$ for every K\"ahler class on $X$. For this, take
any K\"ahler class $[\omega]$.

First assume that $\text{degree}(TX)\,> \, 0$ with respect to $[\omega]$. We will
construct a co--Higgs field on the holomorphic vector bundle
\begin{equation}\label{E}
E\, :=\, {\mathcal O}_X\oplus TX\, .
\end{equation}
Since the vector bundle $\text{Hom}(TX,\, {\mathcal O}_X)\,=\, \Omega_X$ is a direct summand 
$\text{End}(E)$, we have
$$
\text{End}(TX)\,=\, \Omega_X\otimes TX\,=\, \text{Hom}(TX,\, {\mathcal O}_X)\otimes TX
\,\subset\, \text{End}(E)\otimes TX\, .
$$
Hence $\text{Id}_{TX}\, \in\, H^0(X,\, \text{End}(TX))$ is a co--Higgs field on $E$;
this co--Higgs field on $E$ will be denoted by $\theta$.

We will show that the co--Higgs bundle $(E,\, \theta)$ is semistable.

For show that, take any coherent analytic subsheaf $F\, \subset\, E$ such that $\theta(F)\,
\subset\, F\otimes TX$. First consider the case where
$$
F\bigcap (0,\, TX)\, =\, 0\, .
$$
Then the composition
$$
F\, \hookrightarrow\, E \,=\, {\mathcal O}_X\oplus TX\,\longrightarrow\, {\mathcal O}_X
$$
is injective. Hence
$$
\mu(F)\, \leq\, \mu({\mathcal O}_X)\,=\, 0\, <\, \mu(E)\, .
$$
Hence the co--Higgs subsheaf $(F,\, \theta\vert_F)$ of $(E,\, \theta)$ does not violate
the inequality condition for semistability.

Next assume that
$$
F\bigcap (0,\, TX)\, \not=\, 0\, .
$$
Now in view of the given condition that
$\theta(F)\, \subset\, F\otimes TX$, from the construction of the
co--Higgs field $\theta$ is follows immediately that
$$
F\bigcap ({\mathcal O}_X,\, 0)\, \not=\, 0\, .
$$
Hence we have
\begin{equation}\label{fmu}
F\, =\, (F\bigcap (0,\, TX))\oplus (F\bigcap ({\mathcal O}_X,\, 0))\, .
\end{equation}
Note that $$\mu(F\bigcap (0,\, TX))\, \leq\, \mu(TX)$$ because $TX$ is semistable, and
also we have $\mu(F\bigcap ({\mathcal O}_X,\, 0))\,\leq\, \mu({\mathcal O}_X)$. Therefore, from
\eqref{fmu} it follows that
$$
\mu(F)\, \leq\, \mu(E)\, .
$$
Hence again the co--Higgs subsheaf $(F,\, \theta\vert_F)$ of $(E,\, \theta)$ does not violate
the inequality condition for semistability. So $(E,\, \theta)$ is semistable.

Hence by the given condition, the holomorphic vector bundle $E$ is semistable. But this implies
that $\text{degree}(TX)\,=\, 0$. This contradicts the assumption that $\text{degree}(TX)\,> \, 0$.

Now assume that $\text{degree}(TX)\,< \, 0$. We will construct a Higgs field on the
vector bundle $E$ in \eqref{E}.

The vector bundle $\text{Hom}({\mathcal O}_X,\, TX)\,=\, TX$ is a direct summand
$\text{End}(E)$. Hence we have
$$
\text{End}(TX)\,=\, TX\otimes \Omega_X\,=\, \text{Hom}({\mathcal O}_X,\, TX)\otimes\Omega_X
\,\subset\, \text{End}(E)\otimes\Omega_X\, .
$$
Consequently, $\text{Id}_{TX}\, \in\, H^0(X,\, \text{End}(TX))$ is a Higgs field on $E$;
this Higgs field on $E$ will be denoted by $\theta'$.

We will show that the above Higgs vector bundle $(E,\, \theta)$ is semistable.

Take any coherent analytic subsheaf $$F\, \subset\, E$$ such that $\theta(F)\,
\subset\, F\otimes\Omega_X$ and  $\text{rank}(F)\, <\, \text{rank}(E)$.
First consider the case where
$$
F\bigcap ({\mathcal O}_X,\, 0)\, =\, 0\, .
$$
Then we have $F\, \subset\, (0,\, TX)\, \subset\, E$. Since $TX$ is semistable, we have
$$
\mu(F)\, \leq\, \mu(TX)\, .
$$
On the other hand, $\mu(TX)\, <\, \mu(E)$, because $\text{degree}(TX)\,< \, 0\,=\,
\mu({\mathcal O}_X)$. Combining these we get
$$
\mu(F)\, <\, \mu(E)\, ,
$$
and consequently, the Higgs subsheaf $(F,\, \theta\vert_F)$ of $(E,\, \theta)$ does not violate
the inequality condition for semistability.

Now assume that
$$
F\bigcap ({\mathcal O}_X,\, 0)\, \not=\, 0\, .
$$
Hence
\begin{equation}\label{r-i}
\text{rank}(F\bigcap ({\mathcal O}_X,\, 0))\,=\, 1\, ,
\end{equation}
because $F\bigcap ({\mathcal O}_X,\, 0)$ is a nonzero subsheaf of ${\mathcal O}_X$. 
Now from the construction of the Higgs field $\theta$ it follows that
$$
\text{rank}(F\bigcap (0,\, TX))\,=\, \text{rank}(TX)\, .
$$
Combining this with \eqref{r-i} we conclude that $\text{rank}(F)\, =\, \text{rank}(E)$.
This contradicts the assumption that $\text{rank}(F)\, <\, \text{rank}(E)$. Hence we
conclude that the Higgs vector bundle $(E,\, \theta)$ is semistable.

Now the given condition says that $E$ is semistable, which in turn implies
that $$\text{degree}(TX)\,=\, 0\, .$$ This contradicts the assumption that $\text{degree}(TX)\,< \, 0$.

Therefore, we conclude that $\text{degree}(TX)\,=\, 0$ for all K\"ahler classes $[\omega]$ on $X$.
In other words,
\begin{equation}\label{cp}
c_1(TX)\cup ([\omega])^{d-1}\,=\, 0
\end{equation}
for every K\"ahler class $[\omega]$ on $X$,
where $d$ as before is the complex dimension of $d$. But the $\mathbb R$--linear span of
$$
\{[\omega]^{d-1}\,\in\, H^{2d-2}(X,\, {\mathbb R})\, \mid\, [\omega]\ ~ \text{ K\"ahler class}\}
$$
is the full $H^{2d-2}(X,\, {\mathbb R})$. Therefore, from \eqref{cp} it follows that
$$
c_1(TX)\cup \delta\,=\, 0
$$
for all $\delta\, \in\, H^{2d-2}(X,\, {\mathbb R})$. Now from the Poincar\'e duality
it follows that $c_1(TX)\,\in\, H^{2}(X,\, {\mathbb R})$ vanishes.
\end{proof}


\vskip 1cm
\noindent
Indranil Biswas\\
School of Mathematics, Tata Institute of Fundamental
Research 
\\
Homi Bhabha Road, Mumbai 400005, India\\
indranil@math.tifr.res.in

\bigskip
\noindent
{Ugo Bruzzo}\\
{Scuola Internazionale Superiore di Studi Avanzati (SISSA)\\
 Via Bonomea 265, 34136 
Trieste, Italy
\\
 Istituto Nazionale di Fisica Nucleare, Sezione di Trieste}
\\
{bruzzo@sissa.it}

\bigskip
\noindent
{Beatriz Gra\~na Otero}
\\
{Departamento de Matem\'aticas, Pontificia Universidad Javeriana,
\\
Cra. 7$^{\rm ma}$ N$^{\rm o}$ 40-62, Bogot\'a, Colombia}
\\
{bgrana@javeriana.edu.co}

\bigskip
\noindent
{Alessio Lo Giudice}
\\
{Via Antonio Locatelli 20, 37122 Verona}
\\
{alessiologiudic@gmail.com}

\label{lastpage}

\begin{thebibliography}{ZZZZZ}

\bibitem[AB]{AB} B. Anchouche and I. Biswas, Einstein--Hermitian connections on
polystable principal bundles over a compact K\"ahler manifold, {\it Amer. Jour. Math.}
{\bf 123} (2001), 207--228.

\bibitem[At]{At} M. F. Atiyah, Complex analytic connections in fibre
bundles, \textit{Trans. Amer. Math. Soc.} \textbf{85} (1957), 181--207.

\bibitem[BH]{BH} E. Ballico, S. Huh, A note on co-Higgs bundles, {\tt 
arXiv:1606.01843 [math.AG].}

\bibitem[Bi]{Bi} I. Biswas, Yang--Mills connections on compact complex tori,
{\it Jour. Topol. Anal.} {\bf 7} (2015), 293--307.

\bibitem[BBGL]{BBGL} I. Biswas, U. Bruzzo, B. Gra\~na Otero and A. Lo
Giudice, Yang-Mills-Higgs connections on Calabi-Yau manifolds, \textit{Asian
Jour. Math.} (to appear). {\tt arXiv:1412.7738 [math.AG]}.

\bibitem[BF]{BF} I. Biswas and C. Florentino, Commuting elements in reductive groups and Higgs 
bundles on abelian varieties, {\it Jour. Alg.} {\bf 388} (2013), 194--202.

\bibitem[BiSc]{BS} I. Biswas and G. Schumacher, Yang--Mills equation for 
stable Higgs sheaves, {\it Inter. Jour. Math.} {\bf 20} (2009),
541--556.

\bibitem[BiSu]{BiSu} I. Biswas and S. Subramanian, Semistability and finite maps,
{\it Arch. Math.} {\bf 93} (2009), 437--443.

\bibitem[Bo]{Bo} A. Borel, \textit{Linear algebraic groups}, Second edition, Graduate
Texts in Mathematics, 126. Springer-Verlag, New York, 1991. 

\bibitem[FL]{FL} C. Florentino and S. Lawton, The topology of moduli spaces of free group 
representations, {\it Math. Ann.} {\bf 345} (2009), 453--489.

\bibitem[Hi]{Hi} N. Hitchin, {Generalized holomorphic bundles and the B-field action,} \textit{J. 
Geom.  Phys.} {\bf 61} (2011), 352--362.

\bibitem[Hu1]{Hu1} J. E. Humphreys, \textit{Introduction to Lie algebras and representation 
theory,} Graduate Texts in Mathematics, Vol. 9, Springer-Verlag, New York, Heidelberg, Berlin, 
1972.

\bibitem[Hu2]{Hu2} J. E. Humphreys, \textit{Linear algebraic
groups,} Graduate Texts in Mathematics, Vol. 21,
Springer-Verlag, New York, Heidelberg, Berlin, 1987.

\bibitem[Ra1]{Ra1} S. Rayan, {\it Geometry of co-Higgs bundles}, D.Phil. thesis,
Oxford, 2011.

\bibitem[Ra2]{Ra2} S. Rayan, Constructing co-Higgs bundles on $\mathbf{CP}^2$,
\textit{Quart. Jour. Math.} {\bf 65} (2014), 1437--1460.

\bibitem[Si]{Si} C. T. Simpson, Higgs bundles and local systems, 
\textit{Inst. Hautes \'Etudes Sci. Publ. Math.} \textbf{75} (1992),
5--95.

\bibitem[Ya]{Ya} S.-T. Yau, On the Ricci curvature of a compact K\"ahler manifold and
the complex Monge-Amp\'ere equation. I, \textit{Comm. Pure Appl. Math.} {\bf 31}
(1978), 339--411.

\end{thebibliography}
\end{document}